\documentclass[a4paper,11pt]{article}
\usepackage{latexsym}
\usepackage{amssymb}
\usepackage{mathrsfs}
\usepackage{bm}
\usepackage{enumerate}
\usepackage{ stmaryrd }
\usepackage{xcolor}
\usepackage{commath}
\usepackage{array}
\usepackage{amsfonts}
\usepackage{amsmath,amsthm}
\usepackage{hyperref}
\usepackage{algorithm}
\usepackage{algorithmic}
\usepackage{framed}
\usepackage{boites}
\usepackage{graphicx}
\usepackage{cases}
\usepackage{comment}

\makeatletter
\def\moverlay{\mathpalette\mov@rlay}
\def\mov@rlay#1#2{\leavevmode\vtop{%
   \baselineskip\z@skip \lineskiplimit-\maxdimen
   \ialign{\hfil$\m@th#1##$\hfil\cr#2\crcr}}}
\newcommand{\charfusion}[3][\mathord]{
    #1{\ifx#1\mathop\vphantom{#2}\fi
        \mathpalette\mov@rlay{#2\cr#3}
      }
    \ifx#1\mathop\expandafter\displaylimits\fi}
\makeatother

\newcommand{\cupdot}{\charfusion[\mathbin]{\cup}{\cdot}}

\newenvironment{@abssec}[1]{%
    \if@twocolumn

      \section*{#1}%
    \else

      \vspace{.05in}\footnotesize
      \parindent .2in
 {\upshape\bfseries #1. }\ignorespaces
    \fi}

    {\if@twocolumn\else\par\vspace{.1in}\fi}

\newenvironment{keywords}{\begin{@abssec}{\keywordsname}}{\end{@abssec}}

\newenvironment{AMS}{\begin{@abssec}{\AMSname}}{\end{@abssec}}

\newcommand\keywordsname{Key words}
\newcommand\AMSname{AMS subject classifications}
\newcommand\AMname{AMS subject classification}
\newcommand\restr[2]{{
\left.\kern-\nulldelimiterspace 
#1 
\vphantom{|} 
\right|_{#2} 
}}
\newtheorem{theorem}{Theorem}[section]
\newtheorem{lemma}[theorem]{Lemma}
\newtheorem{corollary}[theorem]{Corollary}

\newtheorem{remark}[theorem]{Remark}

\newtheorem{conjecture}[theorem]{Conjecture}
\newtheorem{mainthm}{Theorem}

\newtheorem{thm}{Theorem}

\newcommand{\NN}{\mathbb{N}}

\newcommand{\RR}{\mathbb{R}}

\def\XXint#1#2#3{{\setbox0=\hbox{$#1{#2#3}{\int}$}
\vcenter{\hbox{$#2#3$}}\kern-.5\wd0}}

\newcommand{\link}{\mathop{\circ\kern-.35em -}}

\newcommand{\ol}{\overline}
\newcommand{\pa}{\partial}

\newcommand{\dv}{\mathop{\mathrm{div}}}

\newcommand{\gr}{\nabla}
\newcommand\jump[1]{\left\llbracket #1\right\rrbracket}

\newcommand{\al}{\alpha}
\newcommand{\be}{\beta}

\newcommand{\De}{\Delta}

\newcommand{\la}{\lambda}
\newcommand{\La}{\Lambda}

\newcommand{\te}{\theta}

\newcommand{\Om}{\Omega}
\newcommand{\rn}{{\mathbb{R}}^N}
\newcommand{\sg}{\sigma}

\newcommand\setbld[2]{\left\{ #1 \; :\; #2\right\}}

\newcommand{\tin}{{\text{in }}}
\newcommand{\ton}{{\text{on }}}

\newcommand{\tfor}{{\text{for }}}
\newcommand{\id}{{\rm Id}}

\newcommand{\cdottone}{{\boldsymbol{\cdot}}}

\newcommand{\cB}{\mathcal{B}}
\newcommand{\cC}{\mathcal{C}}

\newcommand{\cJ}{{\mathcal J}}

\newcommand{\cN}{{\mathcal N}}

\newcommand{\cX}{\mathcal{X}}

\numberwithin{equation}{section}

\title{Symmetry and asymmetry in a multi-phase \\ overdetermined problem
\thanks{This research was partially supported by JSPS KAKENHI (under grant nos. JP22K13935 and JP21KK0044, JP23H04459).
}
}

\author{Lorenzo Cavallina 
}
\date{}

\begin{document}

\maketitle

\begin{abstract}
A celebrated theorem of Serrin asserts that one overdetermined condition on the boundary is enough to obtain radial symmetry in the so-called one-phase overdetermined torsion problem. It is also known that imposing just one overdetermined condition on the boundary is not enough to obtain radial symmetry in the corresponding multi-phase overdetermined problem. In this paper, we show that, in order to obtain radial symmetry in the two-phase overdetermined torsion problem, two overdetermined conditions are needed. Moreover, it is noteworthy that this pattern does not extend to multi-phase problems with three or more layers, for which we show the existence of non-radial configurations satisfying countably infinitely many overdetermined conditions on the outer boundary. 
\end{abstract}



\pagestyle{plain}
\thispagestyle{plain}

\section{Introduction}

Let $\NN$ denote the set of positive integers. For some fixed number $m\in \NN$, let us introduce the problem setting and notation related to multi-phase ($m$-layered) elliptic overdetermined problems. 

Let $\Om_k$ ($k\in \{0,1,\dots, m\}$) be a collection of bounded domains of $\rn$ ($N\ge 2$) satisfying
\begin{equation*}
    \emptyset =: \Om_0 \subset\subset \Om_1 \subset \subset \dots \subset \subset \Om_m =:\Om, 
\end{equation*}
where $A\subset\subset B$ means that ``$A$ is compactly contained in $B$", that is, $\ol A\subset B$.
Also, for $k\in \{1,\dots, m\}$, we will assume that the sets $    D_k:=\Om_k\setminus \ol \Om_{k-1} $ are connected.
Moreover, let $\sg$ be the piece-wise constant function defined as 
\begin{equation*}
    \sg:= \sum_{k=1}^m \sg_k \cX_{D_k}, 
\end{equation*}
where $\sg_k$ are positive constants satisfying 
\begin{equation}\label{sg different}
\sg_{k-1}\ne \sg_{k}\quad  \tfor k\in \{2,\dots,m\}.    
\end{equation}
\begin{figure}[h]
\centering
\includegraphics[width=0.6\linewidth]{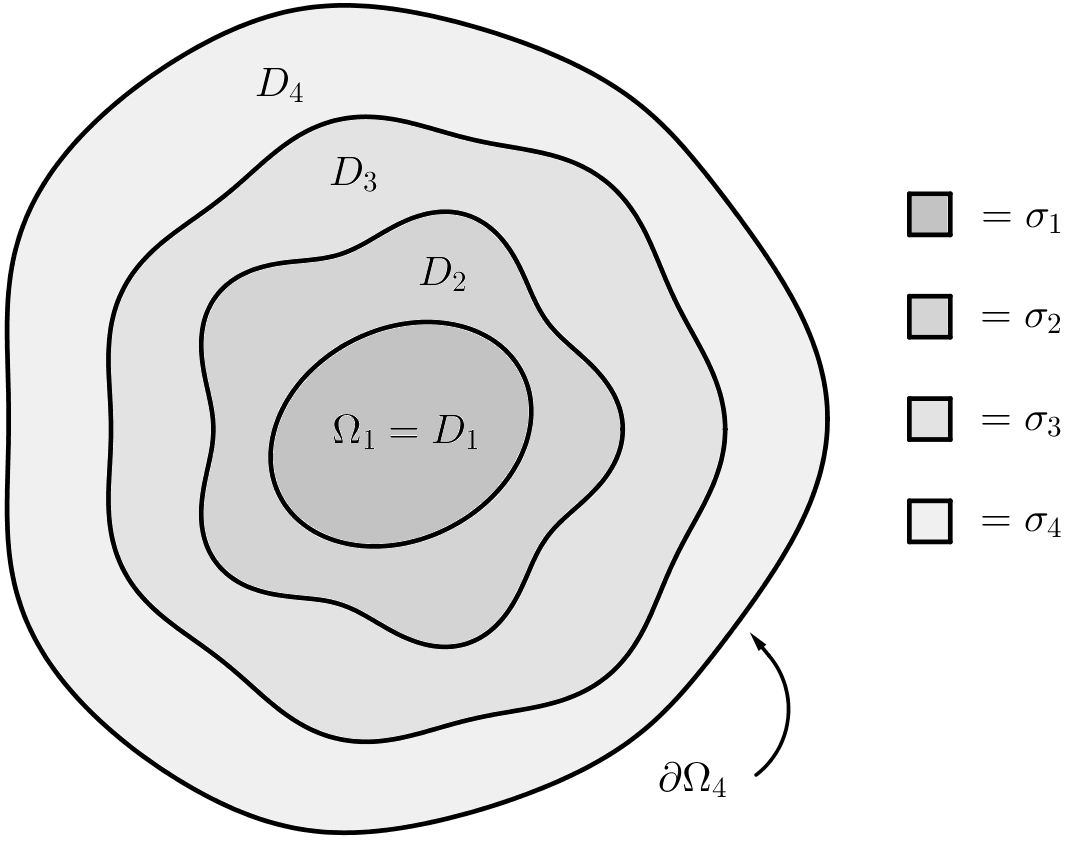}
\caption{Problem setting for $m=4$.} 
\label{fig pbsett}
\end{figure}

Finally, let $n$ denote the outward unit normal vector to $\Om_k$ ($k=1,\dots, m$) and $\pa_n$ be the corresponding normal derivative. Similarly, let $(\pa_n)^j$ denote the so-called, $j$th order normal derivative, which is given by the following expression:
\begin{equation*}
    (\pa_n)^j u(x):= \sum_{|\be|=j} \pa^\be u(x) n^\be(x), 
\end{equation*}
where the summation is taken over all multi-indices $\be=(\be_1,\dots, \be_N)$ of length $k$ and 
\begin{equation*}
    \pa^\be := \frac{\pa^j}{\pa x_1^{\be_1}\dots \pa x_N^{\be_N}},\quad n^\be =(n_1,\dots, n_N)^\be := n_1^{\be_1}\dots n_N^{\be_N}.
\end{equation*}
In this paper, we will consider the following boundary value problem, which will be referred to as the multi-phase torsion problem:
\begin{equation}\label{bvp}
  \begin{cases}
-\dv(\sg \gr u)=1\quad \tin \Om,\\
u=0\quad \ton \pa\Om.
  \end{cases}
\end{equation}
We recall that even when no additional smoothness assumptions are imposed on $\Om_k$, the weak solution to \eqref{bvp} is defined as the unique function $u\in H_0^1(\Om)$ satisfying
\begin{equation}\label{weak formulation}
    \int_\Om \sg \gr u\cdot \gr \varphi = \int_\Om \varphi \quad \text{for all }\varphi\in H_0^1(\Om).
\end{equation}

It is a well-known fact (see \cite{athastra96}) that, under suitable smoothness assumptions, the boundary value problem \eqref{bvp} can be rewritten as the following \emph{transmission problem} (the name \emph{diffraction problem} is also used, see \cite[Chapter 16]{LU}):
\begin{equation}\label{transmission problem}
    \begin{cases}
       -\sg_k \De u =1\quad \tin D_k \quad (k=1,\dots, m),\\
       \jump{u}=0 \quad \ton \pa\Om_k \quad (k=1,\dots, m-1),\\ 
       \jump{\sg\pa_n u}=0 \quad \ton \pa\Om_k \quad (k=1,\dots, m-1),\\ 
       u=0\quad \ton \pa\Om_m,
    \end{cases}
\end{equation}
where the quantity $\jump{\cdottone}$, the \emph{jump} through the interface $\pa\Om_k$, is defined as follows: for any function $f\in H^1(\Om_m)$, we set
\begin{equation*}
    \jump{f}:= \restr{f_{k+1}}{\pa\Om_k}-\restr{f_{k}}{\pa\Om_k} \quad \ton \pa\Om_k,
\end{equation*}
where $f_j:= \restr{f}{D_j}$ ($j=1, \dots, m$).

If the domains $\Om_k$ are concentric balls then, by unique solvability, there exist real constants $\{c_k\}_{k\in \NN}$ such that the solution $u$ of \eqref{bvp} satisfies $(\pa_n)^k u \equiv c_k$ on $\pa\Om$ for all $k$. The aim of this paper is to investigate to what extent the reverse implication holds. It may not be surprising to know that the answer depends on the number of layers $m$. 

The case $m=1$ was solved by Serrin. By adapting the famous \emph{reflection principle} of Alexandrov (see \cite{Ale1958} and Theorem \ref{thm soap bubble} of page \pageref{thm soap bubble}), he showed the following symmetry result: 

\begin{thm}[\cite{Se1971}]\label{thm serrin}
Let $m=1$. Problem \eqref{bvp} admits a solution $u\in C^1(\ol \Om)\cap C^2(\Om)$ satisfying $\pa_n u\equiv c$ on $\pa\Om$ for some constant $c\in\RR$ if and only if $\Om$ is a ball. 
\end{thm}

We refer the interested reader to the survey paper \cite{Magna aswr} for an overview of some qualitative and quantitative results related to Theorems \ref{thm serrin} and \ref{thm soap bubble}.

As shown by the author and Yachimura by making use of a perturbation 
 method relying on the implicit function theorem, \emph{one} overdetermined condition is not enough to obtain symmetry when $m=2$. 

\begin{thm}[\cite{CY1}]\label{thm CY}
Let $m=2$. Then, there exist infinitely many pairs of domains $\Om_1\subset\subset\Om_2$ that are not concentric balls but such that the solution $u$ of \eqref{bvp} satisfies $\pa_n u\equiv c$ on $\pa\Om_2$ for some constant $c\in\RR$.
\end{thm}

Regarding the study of nontrivial solutions to the above two-phase overdetermined problem, further analysis has been carried out concerning local bifurcation  (\cite{CYisaac}) and stability (\cite{CPY}) respectively. 

The difference in behaviors between one-phase and two-phase elliptic overdetermined problems presented by Theorems \ref{thm serrin} and \ref{thm CY} admits the following heuristic interpretation. 
The one-phase overdetermined problem of Theorem \ref{thm serrin} has one constraint (the overdetermined condition) and one degree of freedom (the shape of $\Om$). On the other hand, the two-phase overdetermined problem of Theorem \ref{thm CY} also has one constraint (the overdetermined condition) but two degrees of freedom (the shapes of both $\Om_1$ and $\Om_2$). In other words, Theorem \ref{thm CY} shows that, when the number of degrees of freedom exceeds that of constraints, the overdetermined problem admits nontrivial solutions. 

By combining the soap bubble theorem of Alexandrov \cite{Ale1958} and a symmetry result by Sakaguchi \cite{Sak bessatsu}, we obtain the following result:

\begin{mainthm}\label{mainthm I}
Let $m=2$ and let $\Om_1\subset\subset\Om_2$ be bounded domains of class $C^2$ such that $D_1:=\Om_2\setminus\ol\Om_1$ is connected. Then, problem \eqref{bvp} admits a solution $u$ of class $C^2$ in a neighborhood of $\pa\Om_2$ satisfying $(\pa_n)^k u\equiv c_k$ on $\pa\Om_2$ ($k=1,2$) for some constants $c_1,c_2\in\RR$ if and only if $(\Om_1, \Om_2)$ are concentric balls. 
\end{mainthm}

In light of this result, together with Theorem \ref{thm serrin}, one might be tempted to formulate the following (false!) conjecture.

\begin{conjecture}[False]\label{false}
Let $m\in \NN$ and let $\Om_k$, $k\in \{1,\dots, m\}$, be as in the introduction. Then, problem \eqref{bvp} admits a solution $u$ of class $C^m$ is a neighborhood of $\pa\Om_m$ satisfying 
\begin{equation*}
(\pa_n)^k u\equiv c_k \quad \ton \pa\Om_m \quad (k=1,2,\dots, m)    
\end{equation*}
 for some constants $c_k\in\RR$ if and only if the sets $\Om_k$ are concentric balls.    
\end{conjecture}

We will show that Conjecture \ref{false} does not hold for $m\ge 3$. As a matter of fact, we are able to exhibit a counterexample as follows.

\begin{mainthm}\label{mainthm II}
    Let $m\ge 3$. Then, for all $\sg_1,\dots, \sg_m>0$ satisfying \eqref{sg different} there exist infinitely many domains $\Om_1\subset \subset\dots \subset\subset \Om_m$ where $(\Om_1, \Om_2)$ are not concentric balls but such that the solution $u$ of \eqref{bvp} satisfies
    \begin{equation}\label{infinite overdetermination}
(\pa_n)^k u\equiv c_k \quad \ton \pa\Om_m \quad \forall k\in \NN    
\end{equation}
 for some constants $c_k\in\RR$. 
\end{mainthm}

This paper is organized as follows. In Section \ref{s2}, we show Theorem \ref{mainthm I} while the subsequent sections are devoted to the proof of Theorem \ref{mainthm II}. In Section \ref{s3}, we introduce some preliminary results concerning the shape differentiability of state functions in two-phase problems. In Section \ref{s4}, we study the invertibility properties of a linearized operator (Dirichlet-to-Neumann map). Then, in Section \ref{s5} we combine the results of the two preceding sections to give a proof of Theorem \ref{mainthm II} by means of the implicit function theorem. Finally, Section \ref{s6} is devoted to some comments on the proof of Theorem \ref{mainthm II} and how it relates to the existing literature.

\section{Proof of Theorem \ref{mainthm I}}\label{s2}

In this section, we will give a simple proof of Theorem \ref{mainthm I}. To this end, we will need the following symmetry results. 

\begin{thm}[\cite{Ale1958}]\label{thm soap bubble}
A compact hypersurface, embedded in $\rn$, that has constant mean curvature must be a sphere.     
\end{thm}

\begin{thm}[\cite{Sak bessatsu}]\label{thm bessatsu}
Let $m=2$. Let $\Om_2$ be an open ball and let $\Om_1\subset\subset\Om_1$ be a 
 bounded open set of class $C^2$ with finitely many connected components such that $D_1:=\Om_2\setminus\ol\Om_1$ is connected. Then, problem \eqref{bvp} admits a solution $u$ satisfying $\pa_n u\equiv c_1$ on $\pa\Om_2$ for some real constant $c_1$ if and only if $(\Om_1,\Om_2)$ are concentric balls.
\end{thm}

We remark that both theorems above were originally stated in a more general setting (see also \cite{Ale1962}) but the formulations above are enough for our purposes.

Theorem \ref{mainthm I} now  follows by combining the two theorems above. 

\begin{proof}[Proof of Theorem \ref{mainthm I}]
Let $m=2$ and let the sets $\Om_1$ and $\Om_2$ satisfy the hypotheses of the theorem. In what follows, we will assume that problem \eqref{bvp} admits a solution $u$ of class $C^2$ in a neighborhood of $\pa\Om_2$ satisfying 
\begin{equation}\label{hypothesis}
    \pa_n u \equiv c_1 , \quad (\pa_n)^2u\equiv c_2 \quad \ton \pa\Om_2
\end{equation}
for some real constants $c_1,c_2$ and then show that $(\Om_1,\Om_2)$ must be concentric balls. The reverse implication is trivial and therefore omitted. 

Since the solution $u$ is of class $C^2$ in a neighborhood of $\pa\Om_2$, the decomposition formula for the Laplace operator (\cite[Proposition 5.4.12]{HP2018}) combined with \eqref{hypothesis} yields:
\begin{equation}\label{decompose the laplacian}
    \frac{-1}{\sg_2}= \De u = (\pa_n)^2 u + H \pa_n u + \De_\tau u = c_2 + Hc_1 \quad \ton \pa\Om_2,
\end{equation}
where $\De_\tau=\dv_\tau\circ \gr_\tau$ denotes the tangential Laplacian (that is, the tangential divergence of the tangential gradient,  also known as ``Laplace--Beltrami operator", see \cite[Definitions 5.4.5, 5.4.6 and 5.4.11]{HP2018}) on $\pa\Om_2$ and $H$ is the (additive) mean curvature given by the tangential divergence of the outward unit normal $n$ (notice that, under this definition, the mean curvature of a ball of radius $R$ is $\frac{N-1}{R}$).

The terms in \eqref{decompose the laplacian} can be rearranged to show that the mean curvature $H$ is constant on the entire $\pa\Om_2$. Thus, by applying Theorem \ref{thm soap bubble} to each connected component of $\pa\Om_2$ we obtain that $\pa\Om_2$ is the disjoint union of a finite number of spheres with the same radius and orientation. This leaves us with just one possibility, that is, $\pa\Om_2$ is a sphere and $\Om_2$ is a ball. The conclusion readily follows from Theorem \ref{thm bessatsu}.
\end{proof}

\begin{remark}
Notice that in the proof above we did not use the connectedness of $\Om_1$ (nor that of $\pa\Om_2$) but just that of $\Om_2$ and $D_1=\Om_2\setminus\ol\Om_1$. 
\end{remark}

In \cite{multiphase 1, multiphase 2}, the authors showed the radial symmetry of the solutions to a similar multi-phase overdetermined problem in $\rn$ (in the elliptic and parabolic settings respectively),  where the overdetermined condition considered requires the solution $u$ to be constant of each interface. 
In our setting, the following analogous result holds:

\begin{corollary}
Let $m\in\NN$ and let $\Om_k$, $k\in \{1,\dots,m\}$, be as in the introduction. If the solution $u$ to \eqref{transmission problem} satisfies 
\begin{equation}\label{odc}
\begin{aligned}
    u&\equiv \al_k \quad \ton \pa\Om_k \quad (k=1,\dots, m-1), \\ 
    \pa_n u &\equiv c \quad \ton \pa\Om_m,
\end{aligned}
\end{equation}
for some real constants $\al_1, \dots, \al_{m-1}$ and $c$, then the sets $\Om_k$ are concentric balls and the function $u$ is radial.
\end{corollary}
\begin{proof}
We will show the claim by induction on the number of layers $m$. The base case $m=1$ is exactly Serrin's result, Theorem \ref{thm serrin}. In what follows, let us assume that the claim holds when the number of layers is strictly less than $m$ and then show that the claim holds for $m$ as well. Let $u$ be the solution to \eqref{transmission problem} with $m$ layers $\Om_1\subset\subset \dots \subset\subset \Om_m$ and assume that $u$ satisfies the overdetermined conditions \eqref{odc}. Let
\begin{equation*}
    \widetilde u := \begin{cases}
        \frac{\sg_1}{\sg_2}(u-\al_1) +\al_1 \quad \tin \Om_1,\\
        u \quad \tin \Om_{m}\setminus \Om_1.
    \end{cases}
\end{equation*}
By construction, $\widetilde u$ solves the transmission problem \eqref{transmission problem} with $m-1$ layers $\Om_2\subset\subset \dots \subset\subset \Om_{m}$. Moreover, $\widetilde u$ also satisfies the overdetermined conditions \eqref{odc} (starting from the ``first" interface $\pa\Om_2$). Thus, by the inductive hypothesis, $\Om_2,\dots, \Om_m$ are concentric balls and $\widetilde u$ is radial. As a consequence, $u$ is also radial and, since $\pa\Om_1$ is a level set of $u$, the remaining set $\Om_1$ is also a ball concentric with $\Om_m$. This concludes the proof.
\end{proof}

\section{Preliminaries on shape derivatives}\label{s3}

In this section, we are going to introduce the main definitions and known results concerning shape calculus for two-phase problems that are going to be useful in the proof of Theorem \ref{mainthm II}. The experienced reader might therefore skip this section. 

In what follows, let $\Om_1$ and $\Om_2$ be concentric balls of radii $R\in (0,1)$ and $1$ respectively. Also, without loss of generality, suppose that $\sg_2:=1$.  
Let $\al\in(0,1)$. 
For sufficiently small $\eta\in C^{2,\al}(\pa\Om_1)$ and $\xi\in C^{2,\al}(\pa\Om_2)$, let $D_\eta$ and $\Om_\xi$ be the bounded domains whose boundaries are given by 
\begin{equation}\label{D_eta Om_xi}
    \pa D_\eta:=\setbld{x+\eta(x) n(x)}{x\in \pa\Om_1}, \quad
        \pa \Om_\xi:=\setbld{x+\xi(x) n(x)}{x\in \pa\Om_2}.
\end{equation}

Let $v_{\xi,\eta}$ be the solution to the following two-phase boundary value problem associated to the pair $(D_\eta,\Om_\xi)$:
\begin{equation}\label{v_eta xi}
    \begin{cases}
        -\dv(\sg_{\xi,\eta} \gr v_{\xi,\eta})=1, \quad \tin \Om_\xi,\\
        v_{\xi,\eta}=f \quad \ton \pa\Om_\xi,
    \end{cases}
\end{equation}
where $\sg_{\xi,\eta}:=\sg_1\cX_{D_\eta}+\sg_2 \cX_{\Om_\xi\setminus \ol D_\eta}$ and $f\in C^{2,\al}(\rn)$ is a given function. 

The machinery of shape derivatives is the right tool to give a quantitative description of how $v_{\xi,\eta}$ depends on the perturbations $(\xi,\eta)$.
The main technical difficulties lie in the following two points: firstly, the functions $v_{\xi,\eta}$ depend on two parameters, and secondly, each $v_{\xi,\eta}$ lies in a different function space depending on the pair $(\xi,\eta)$.
To overcome these difficulties, let $\Theta:=C^{2,\al}(\rn,\rn)$ and consider the following construction.
For small $\te\in \Theta$, set 
\begin{equation*}
    D_\te:=(\id+\te)(\Om_1), \quad 
    \Om_\te:=(\id+\te)(\Om_2), \quad \sg_\te:= \sg_1\cX_{D_\te}+\cX_{\Om_\te\setminus\ol D_\te}
\end{equation*}
and let $v_\te$ be the unique solution to \eqref{v_eta xi} with respect to the pair $(D_\te,\Om_\te)$. Moreover, set 
\begin{equation}\label{pulled-back function}
    V(\te):= v_\te\circ(\id+\te)\in H^1(\Om_2),\quad \text{for small $\te\in \Theta$}. 
\end{equation}
Then, the (first-order) \emph{shape derivative} of $v_\te$ at $\te=0$ is defined as 
\begin{equation}\label{shape derivative}
    v'[\te]:=V'(0)[\te]-\gr V(0)\cdot \te, \quad \tfor \te\in \Theta,
\end{equation}
where $V'(0)[\theta]$ denotes the Fr\'echet derivative of $V$ at $\te=0$ in the direction $\te\in \Theta$. Notice that the definition \eqref{shape derivative} is given in such a way as to be compatible with a formal application of partial differentiation with respect to $\te$ in \eqref{pulled-back function}.

\begin{lemma}\label{lemma shape derivative}
We have the following:
\begin{enumerate}[(i)]
    \item     The map $\te\mapsto V(\te)\in H^1(\Om_2)$ is Fr\'echet differentiable in a neighborhood of $0\in \Theta$. 
    \item Let $U$ be a neighborhood of $\pa\Om_2$ that does not intersect $\ol\Om_1$ and set $K:= \ol U \cap \ol \Om_2$. Then, $\te\mapsto \restr{V(\te)}{K}\in C^{2,\al}(K)$ is Fr\'echet differentiable in a neighborhood of $0\in \Theta$. 
    \item Let $E: C^{2,\al}(\pa\Om_2)\times C^{2,\al}(\pa\Om_1)\to \Theta$ be a bounded linear extension operator such that 
    \begin{equation*}
        \restr{E(\xi,\eta)}{\pa\Om_1}=\eta n, \quad \restr{E(\xi,\eta)}{\pa\Om_2}=\xi n \quad \tfor (\xi,\eta)\in C^{2,\al}(\pa\Om_2)\times C^{2,\al}(\pa\Om_1). 
    \end{equation*}
    Following \eqref{pulled-back function}, set $V(\xi, \eta):=V(E(\xi,\eta))$. Then, the mappings $(\xi,\eta)\mapsto V(\xi,\eta)\in H^1(\Om_2)$ and $(\xi,\eta)\mapsto \restr{V(\xi,\eta)}{K}\in C^{2,\al}(K)$ are Fr\'echet differentiable in a neighborhood of $(0,0)\in C^{2,\al}(\pa\Om_2)\times C^{2,\al}(\pa\Om_1)$. 
    \item Following \eqref{shape derivative}, let $v'[\xi]:=v'[E(\xi,0)]$ denote the shape derivative of $v_\te$ with respect to the outer perturbation $\xi$ only. Then, $v'[\xi]$ is independent of the extension operator $E$ and can be characterized as the unique solution to the following boundary value problem:
\begin{equation}\label{general shape derivative}
    \begin{cases}
        -\dv(\sg \gr v'[\xi])=0\quad\tin \Om_2,\\
        v'[\xi]=(\pa_n f - \pa_n V(0))\xi\quad \ton \pa\Om_2.
    \end{cases}
\end{equation}    
\end{enumerate}
\end{lemma}
\begin{proof}[Sketch of the proof]
In the case of the Laplace operator ($\sg\equiv 1$) with homogeneous Dirichlet boundary conditions ($f\equiv 0$), the claims $(i)-(iv)$ are well-known results that can be obtained by a standard procedure that combines the implicit function theorem (Theorem \ref{ift}, page \pageref{ift}) and the Schauder regularity theory \cite[Chapter 6]{GT}: see, for instance, \cite[Section 5.3 and the final remark therein]{HP2018}. Also, the case of two-phase boundary value problems with homogeneous Dirichlet boundary conditions has been dealt with in \cite[Appendix]{cava sap}, while the case of general boundary conditions for the Laplace operator has been briefly covered in \cite[Section 5.6]{HP2018} and the references therein. Finally, as far as the extension operator $E$ is concerned, we refer the interested reader to \cite[Section 6.9]{GT}.

Here, we will just limit to showing a simple trick that allows us to reduce to the cases mentioned above. 
For small $\te\in \Theta$, let $v_\te$ be the solution to \eqref{v_eta xi} with respect to the pair $(D_\te, \Om_\te)$ and let 
\begin{equation*}
    w_\te:=v_\te-f\in H_0^1(\Om_\te)\cap C^{2,\al}(K_\te),\quad W(\te):= w_\te\circ(\id+\te)\in H^1_0(\Om_2)\cap C^{2,\al}(K),
\end{equation*}
where $K_\te:=(\id+\te)(K)$. Notice that, by construction, $w_\te$ is a weak solution to
\begin{equation*}
  \begin{cases}
     -\dv(\sg_\te \gr w_\te) = F_\te \quad\tin \Om_\te,\\
     w_\te =0\quad \ton \pa\Om_\te,
  \end{cases}  
\end{equation*}
where the function $F_\te\in H^{-1}(\Om_\te)\cap C^{2,\al}(K_\te)$ is given by $F_\te:= 1+\dv(\sg_\te\gr f)$. Also, notice that, by construction, $F_\te$ admits an extension to $\rn$ that is independent of $\xi$. 
Therefore, all results $(i)-(iv)$ hold for $w_\te$. Since $v_\te=w_\te + f$ by definition, it is clear that $(i)-(iii)$ hold for $v_\te$ as well. To conclude, we just need to check that $v'[\xi]$ solves \eqref{general shape derivative}. To this end, notice that, by definition, we have $w_\te= v_\te-f$. As a result, $w'[\xi]=v'[\xi]$ and thus $v'[\xi]$ satisfies the equation in \eqref{general shape derivative}. 
On the other hand, notice that $W(0)=V(0)-f$ holds by construction, while, by $(iv)$, $w'[\xi]$ satisfies the boundary condition 
\begin{equation*}
w'[\xi]= -\pa_n W(0)\xi \quad \ton \pa\Om_2.    
\end{equation*}
The claim $(iv)$ for $v_\te$ follows by combining the identities above.
\end{proof}

\section{The two-phase Dirichlet-to-Neumann map}\label{s4}
As in the previous section, let $\Om_1$ and $\Om_2$ be concentric balls of radii $R\in (0,1)$ and $1$ respectively and let $\sg_2:=1$. 
Let us introduce the following two-phase Dirichlet-to-Neumann map $\cN:C^{2,\al}(\pa\Om_2)\to C^{1,\al}(\pa\Om_2)$ defined as $\xi\mapsto \pa_n w[\xi]$, where $w[\xi]$ is the unique solution to the following boundary value problem 
\begin{equation}\label{w xi}
    \begin{cases}
        -\dv(\sg \gr w)=0 \quad \tin \Om_2,\\
        w=\xi \quad \ton \pa\Om_2. 
    \end{cases}
\end{equation}
Let $\{Y_{k,i}\}_{k,i}$ ($k\in \{0,1,\dots\}$, $i\in\{1,2,\dots, d_k\}$) be a maximal family of linearly independent solutions to the eigenvalue problem
\begin{equation*}
-\De_{\tau} Y_{k,i}=\lambda_k Y_{k,i} \quad\textrm{ on } \pa\Om_2,
\end{equation*}
with $k$-th eigenvalue $\lambda_k=k(N+k-2)$ of multiplicity $d_k$ and normalized in such a way that $\norm{Y_{k,i}}_{L^2(\pa\Om_2)}=1$. The functions $\{Y_{k,i}\}$ are usually referred to as spherical harmonics. 
By the method of separation of variables, it can be shown that the spherical harmonics form an orthonormal basis of eigenfunctions of $\cN$ in $L^2(\pa\Om_2)$. The eigenvalues of $\cN$ have been computed in \cite{CY1}.
\begin{lemma}\label{lem eigenvalues}
For $k\in \NN\cup\{0\}$ and $i\in \{1,\dots, d_k\}$, we have
\begin{equation}\label{cN eigenvalues}
    \cN (Y_{k,i}) = k \frac{(2-N-k)(1-\sg_1)+(N-2+k+k\sg_1)R^{2-N-2k}}{F} \, Y_{k,i},
\end{equation}
where 
\begin{equation*}
    F:=k(1-\sg_1)+(N-2+k+k \sg_1)R^{2-N-2k}>0.
\end{equation*}
Moreover, $\cN(Y_{k,i})=0$ if and only if $k=0$. 
\end{lemma}
\begin{proof}
Since the eigenvalues of \eqref{cN eigenvalues} have been computed in \cite{CY1}, in what follows we only need to check that the right-hand side in \eqref{cN eigenvalues} vanishes if and only if $k=0$. 
Furthermore, since $\cN(Y_{0,i})=0$ by construction, it will suffice to show that $\cN(Y_{k,i})\ne 0$ for $k\in \NN$.
To this end, let $\varphi$ denote the numerator in the right-hand side of \eqref{cN eigenvalues}, that is 
\begin{equation*}
    \varphi(R):= (2-N-k)-\sg_1(2-N-k)+(N-2+k+k\sg_1)R^{2-N-2k} . 
\end{equation*}
We will show that $\varphi(R)>0$ for all $R\in (0,1]$ and $k\in \NN$, proving the claim. First, notice that $\varphi$ is a decreasing function of $R$. Thus, for all $R\in (0,1]$,
\begin{equation*}
    \varphi(R)\ge \varphi(1)= (2-N-k)-\sg_1(2-N-k)+(N-2+k+k\sg_1)= \sg_1 (2k-2+N)>0,
\end{equation*}
which is what we wanted to show.
\end{proof}

Let  $C_*^{i,\al}(\pa\Om_2)$ denote the set of all functions in $C^{i,\al}(\pa\Om_2)$ with zero average over $\pa\Om_2$ ($i\in\{1,2\}$).
Notice that, by Lemma \ref{lem eigenvalues}, $\cN$ fixes the eigenspaces of the Laplace-Beltrami operator, whence $\cN$ is a well-defined operator from $C_*^{2,\al}(\pa\Om_2)$ into $C_*^{1,\al}(\pa\Om_2)$. 
Also by Lemma \ref{lem eigenvalues}, $\cN$ is injective. 
Actually, it can be shown that $\cN:C_*^{2,\al}(\pa\Om_2)\to C_*^{1,\al}(\pa\Om_2)$ is a bijection. In order to show this, we first need the following Lemma. 
\begin{lemma}\label{id+N is a bijection}
    The map $\id+\cN : C^{2,\al}(\pa\Om_2)\to C^{1,\al}(\pa\Om_2)$ is a bijection.
\end{lemma}
\begin{proof}
We will show that, for all $\eta\in C^{1,\al}(\pa\Om_2)$ there exists a unique $\xi\in C^{2,\al}(\pa\Om_2)$ that satisfies
\begin{equation*}
\cN\xi+\xi=\eta.
\end{equation*}
First of all, let us consider the Sobolev space $H^1(\Om_2)$ endowed with the (equivalent) norm $\norm{\psi}_{H^1(\Om_2)}:=\norm{\gr\psi}_{L^2(\Om_2)}+\norm{\restr{\psi}{\pa\Om_2}}_{L^2(\pa\Om_2)}$ and consider the bilinear form $\cB: H^1(\Om_2)\times H^1(\Om_2)\to\RR$ given by
\begin{equation*}
\cB(\psi,\phi):= \int_{\Om_2} \sg \gr \psi\cdot \gr \phi 
+ \int_{\pa\Om_2} \psi\phi.
\end{equation*}
Notice that, by construction, $\cB$ is bilinear, continuous, and coercive.
Fix now an element $\eta\in C^{1,\al}(\pa\Om_2)\subset L^2(\pa\Om_2)$. By the Lax-Milgram theorem, there exists a unique function $w\in H^1(\Om_2)$ such that
\begin{equation}\label{B=<eta phi>}
\cB(w,\phi)= \left\langle \eta,\restr{\phi}{\pa\Om_2}\right\rangle_{L^2(\pa\Om_2)}\quad \text{for all }\phi\in H^1(\Om_2).
\end{equation}
Now, if we restrict the identity above to $\phi$ in $H^1_0(\Om_2)$, then we realize that $w$ must satisfy
\begin{equation}\label{w weak form}
\int_{\Om_2} \sg \gr w\cdot \gr \phi 
=0 \quad \text{for all }\phi\in H_0^1(\Om_2).
\end{equation}
In other words, $w$ satisfies the equation
\begin{equation}\label{w eq}
-\Delta w=0\quad \tin \Om_1 \cup (\Om_2\setminus \ol \Om_1)
\end{equation}
and the transmission conditions:
\begin{equation}\label{w tc}
\jump{w}=0, \quad \jump{\sg\pa_n w}=0 \quad \ton \pa \Om_1. 
\end{equation}
Moreover, integration by parts in \eqref{B=<eta phi>} and the arbitrariness of the trace of $\phi\in H^1(\Om_2)$ on $\pa\Om_2$ yield
\begin{equation}\label{w bc}
\pa_n w + w = \eta \quad \ton \pa\Om_2.
\end{equation}
Now, since $w$ is the solution to the transmission problem \eqref{w eq}, \eqref{w tc} and \eqref{w bc}, we can inductively bootstrap its regularity in a classical way by means of the standard elliptic regularity estimates \cite[Chapter 8]{GT} and the Schauder boundary estimates \cite[Chapter 6]{GT} (see for example the argument in the proof of \cite[Proposition 5.2]{kamburov sciaraffia} after $(5.7)$). We obtain that $w$ is of class $C^{2,\al}$ in an open neighborhood of $\pa\Om_2$ (whose closure does not intersect $\ol \Om_1$). In particular, the function \begin{equation}\label{xi=psi/c}
\xi:= \restr{w}{\pa\Om_2}
\end{equation}
is a well defined element of $C^{2,\al}(\pa\Om_2)$. This, together with \eqref{w eq}, \eqref{w tc} and \eqref{w bc}, implies that $w$ is the solution to \eqref{w xi}.
In particular, by \eqref{xi=psi/c} and \eqref{w bc},
\begin{equation*}
\cN\xi+\xi = \pa_n w + w =\eta \quad \ton \pa\Om_2.
\end{equation*}
By the arbitrariness of $\eta\in C^{1,\al}(\pa\Om_2)$, the above shows that $\id+\cN: C^{2,\al}(\pa\Om_2)\to C^{1,\al}(\pa\Om_2)$ is surjective. Injectivity follows from the coercivity of $\cB$. As a result, $\id+\cN: C^{2,\al}(\pa\Om_2)\to C^{1,\al}(\pa\Om_2)$ is a bijection (whose inverse is continuous by the bounded inverse theorem).    
\end{proof}

\begin{lemma}\label{N is a bijection}
The operator $\cN$ is a bijection from $C_*^{2,\al}(\pa\Om_2)$ into $C_*^{1,\al}(\pa\Om_2)$.
\end{lemma}
\begin{proof}
Let $K:C^{1,\al}(\pa\Om_2)\to C^{2,\al}(\pa\Om_2)\hookrightarrow
C^{1,\al}(\pa\Om_2)$ denote the inverse of $\id+\cN$ (that exists by Lemma \ref{id+N is a bijection}). Notice that, by the compactness of the embedding $C^{2,\al}(\pa\Om_2)\hookrightarrow
C^{1,\al}(\pa\Om_2)$, $K$ is a compact operator from $C^{1,\al}(\pa\Om_2)$ into itself. 
By Lemma \ref{lem eigenvalues} and the Fredholm alternative (Riesz--Schauder theory) \cite[Theorem 6.6 (c)]{Brezis}, $\id-K$ admits a continuous inverse $T:C_*^{1,\al}(\pa\Om_2)\to C_*^{1,\al}(\pa\Om_2)$. Thus, for $(\xi,\eta)\in C_*^{2,\al}(\pa\Om_2)\times C_*^{1,\al}(\pa\Om_2)$ we have
\begin{equation*}
\cN\xi=\eta \iff \xi +\cN\xi = \xi+\eta \iff \xi= K(\xi+\eta) \iff (\id -K)\xi = K \eta \iff \xi = TK\eta. 
\end{equation*}
In other words, the operator $\cN: C_*^{2,\al}(\pa\Om_2)\to C_*^{1,\al}(\pa\Om_2)$ admits a continuous inverse, given by ${\cN}^{-1}=T\circ K$. 
\end{proof}

\section{Proof of Theorem \ref{mainthm II}}\label{s5}

Let $m\ge 3$. For $k=1,\dots, m$, let $\Om_k$ be the open ball of radius $R_k>0$ centered at the origin. Also assume that $0<R_{k}<R_{k+1}$ for $k=1,\dots, m-1$. Moreover, unless otherwise specified, we will always assume $R_2:=1$, $\sg_2:=1$ (notice that this does not result in a loss of generality).

In what follows, we will show Theorem \ref{mainthm II}. In particular, we will find a nontrivial collection of domains 
\begin{equation}\label{nontrivial domains}
    D_\eta\subset\subset \Om_\xi\subset\subset \Om_3 \subset\subset \dots \subset\subset \Om_m
\end{equation}
such that \eqref{infinite overdetermination} is satisfied.

To this end, we will employ the following version of the implicit function theorem for Banach spaces. 
(see \cite[Theorem 2.3, page 38]{AP1983} for a proof). 
\begin{thm}[Implicit function theorem]\label{ift}
Let $\Psi\in\cC^k(X\times \La,Z)$, $k\in\NN$, where $Z$ is a Banach space and $X$ (resp. $\La$) is an open set of a Banach space $\widetilde X$ (resp. $\widetilde \La$). Suppose that  $\Psi(x^*,\la^*)=0$ and that the partial derivative $\pa_x\Psi(x^*,\la^*)$ is a bounded invertible linear transformation from $X$ to $Z$. 

Then, there exist neighborhoods $\La'$ of $\la^*$ in $\widetilde \La$ and $X'$ of $x^*$ in $\widetilde X$, and a map $\xi\in\cC^k(\La',X')$ such that the following hold:
\begin{enumerate}[(i)]
\item $\Psi(\xi(\la),\la)=0$ for all $\la\in\La$,
\item If $\Psi(x,\la)=0$ for some $(x,\la)\in X'\times \La'$, then $x=\xi(\la)$,
\item $\xi'(\la)=-\left(\pa_x \Psi(p) \right)^{-1}\circ \pa_\la \Psi(p)$, where $p=(\xi(\la),\la)$ and $\la\in\La'$.
\end{enumerate}
\end{thm}

Before giving the proof of Theorem \ref{mainthm II}, let us first define some auxiliary functions.
Set $R_0:=0$ and let $u_0$ denote the following radial function:
\begin{equation}\label{u_0}
  u_0(x):= \begin{cases}
     \frac{1}{2N} \left( \sum_{j=k+1}^{m-1} \frac{1}{\sg_{j+1}} (R_{j+1}^2-R_j^2) +\frac{1}{\sg_{k+1}}(R_{k+1}^2-|x|^2)  \right), \\  \text{if }|x|\in [R_k,R_{k+1}], \quad (k=0,\dots, m-1),
  \end{cases}  
\end{equation}
where the value of the sum $\sum_{j=k+1}^{m-1}$ is set to be zero if $m-1<k+1$ (empty sum). 
It is easy to check that $u_0$ solves the transmission problem \eqref{transmission problem}. The following auxiliary function will also play a crucial role in our construction:
\begin{equation}\label{v_0}
  v_0(x):= \begin{cases}
     \frac{1}{2N} \left( \sum_{j=3}^{m-1} \frac{1}{\sg_{j+1}} (R_{j+1}^2-R_j^2) +\frac{1}{\sg_{3}}(R_{3}^2-|x|^2)  \right), \quad \text{if }|x|\in [0,R_{3}],\\
     u_0(x), \qquad \text{if }|x|\in (R_3, R_m]. 
  \end{cases}  
\end{equation}
This is nothing but the solution to problem \eqref{transmission problem} in the case $\sg_1=\sg_2=\sg_3$.

Let $\al\in(0,1)$. 
For sufficiently small $\eta\in C^{2,\al}(\pa\Om_1)$ and $\xi\in C^{2,\al}(\pa\Om_2)$, let $D_\eta$ and $\Om_\xi$ be the bounded domains whose boundaries are given by \eqref{D_eta Om_xi}. Moreover, let $v_{\xi,\eta}$ be the solution to the following boundary value problem associated to the pair $(D_\eta,\Om_\xi)$:
\begin{equation}\label{v_eta xi 2}
    \begin{cases}
        -\dv(\sg_{\xi,\eta} \gr v_{\xi,\eta})=1, \quad \tin \Om_\xi,\\
        v_{\xi,\eta}=v_0 \quad \ton \pa\Om_\xi,
    \end{cases}
\end{equation}
where $\sg_{\xi,\eta}:=\sg_1\cX_{D_\eta}+\sg_2 \cX_{\Om_\xi\setminus \ol D_\eta}$. Notice that $\restr{v_{\xi,\eta}}{(\xi,\eta)=(0,0)}=\restr{u_0}{\Om_2}$.

Lemma \ref{lemma shape derivative} yields 
\begin{equation}\label{v' eq}
    \begin{cases}
        -\dv(\sg\gr v'[\xi])=0, \quad \tin \Om_2,\\
        v'[\xi]= (\pa_n v_0-\pa_n u_0)\xi = \frac{1}{N}(1-\frac{1}{\sg_3})\xi \quad \ton \pa\Om_2.
    \end{cases}
\end{equation}
 
\begin{proof}[Proof of Theorem \ref{mainthm II}]
We will construct nontrivial domains $D_\eta$ and $\Om_\xi$ with $\eta\in C^{2,\al}(\pa\Om_1)$ and $\xi\in C^{2,\al}(\pa\Om_2)$ such that the solution $v_{\xi,\eta}$ to \eqref{v_eta xi 2} satisfies
\begin{equation}\label{what we want}
   \pa_{n_\xi} v_{\xi,\eta} = \sg_3 \pa_{n_\xi} v_0 \quad \ton \pa\Om_\xi,  
\end{equation}
where $\pa_{n_\xi}$ is the normal derivative in the outward direction on $\pa\Om_\xi$.
In other words, if \eqref{what we want} holds, then the function 
\begin{equation}
   v:=\begin{cases}
       v_{\xi,\eta} \quad \tin \Om_\xi,\\
       v_0 \quad \tin \Om_m\setminus \Om_\xi
   \end{cases}   
\end{equation}
solves the transmission problem \eqref{transmission problem} with respect to \eqref{nontrivial domains}. Moreover, since the function $v$ defined above is radial in $\Om_m\setminus \Om_\xi$, in particular, it satisfies \eqref{infinite overdetermination}.
We are therefore left with the problem of finding a nontrivial pair of functions $(\xi,\eta)$ such that \eqref{what we want} holds. 
To this end, consider the following mapping:
\begin{equation}\label{Psi}
\begin{aligned}
 \Psi:\; & C_*^{2,\al}(\pa\Om_2)\times C^{2,\al}(\pa\Om_1) \longrightarrow C_*^{1,\al}(\pa\Om_2),\\
&(\xi,\eta) \mapsto \Big( \big( \pa_{n_\xi}v_{\xi,\eta}-\sg_3\pa_{n_\xi} v_0 \big )\circ (\id+\xi n) \Big) J_\tau(\xi).
\end{aligned}
\end{equation}
Let us first clarify the notation employed in the definition of $\Psi$. Here, $n$ stands for the outward unit normal vector to the unperturbed boundary $\pa\Om_2$ (we recall that the outer normal to $\pa\Om_\xi$ is denoted by $n_\xi$) and $\id+\xi n$ is the natural pullback mapping from $\pa\Om_\xi$ to $\pa\Om_2$. Moreover, $J_\tau(\xi)$ is the tangential Jacobian  associated with the mapping $\id+\xi n$ (that is, it is the multiplicative term that appears in the integrand of a surface integral after the corresponding change of variables, see \cite[(5.67)--(5.68)]{HP2018}). It is known that both $n_\xi$ and $J_\tau(\xi)$ are differentiable with respect to $\xi$ at $\xi=0$ (see \cite[Proposition 5.4.14 and Lemma 5.4.15]{HP2018}. We remark that in \cite[Proposition 5.4.14]{HP2018} only G\^ateaux differentiability is shown. This notwithstanding, the Fr\'echet differentiability of the normal can be shown analogously or by noticing that a smooth extension of $n$ can be written as the normalized gradient of some subharmonic function vanishing on the boundary). These facts together with Lemma \ref{lemma shape derivative} imply that $\Psi$ is a well-defined and Fr\'echet differentiable mapping from a neighborhood of $(0,0)\in C_*^{2,\al}(\pa\Om_2)\times C^{2,\al}(\pa\Om_1)$ into $C^{1,\al}(\pa\Om_2)$. We just need to check that, for small $(\xi,\eta)\in C_*^{2,\al}(\pa\Om_2)\times C^{2,\al}(\pa\Om_1)$, the image $\Psi(\xi,\eta)$ is indeed a function of zero average over $\pa\Om_2$. 
To this end, notice that 
\begin{equation*}
    \dv(\sg_{\xi,\eta} \gr v_{\xi,\eta})=-1=\sg_3 \De v_0 \quad \tin \Om_\xi
\end{equation*}
and thus the divergence theorem and a change of variables yield the desired identity:
\begin{equation*}
    \begin{aligned}
\int_{\pa\Om_2} \pa_{n_\xi}v_{\xi,\eta}\circ(\id+\xi n) J_\tau (\xi) = \int_{\pa\Om_\xi} \pa_{n_\xi} v_{\xi,\eta} = \int_{\Om_\xi} \dv(\sg_{\xi,\eta} \gr v_{\xi,\eta})\\
=-|\Om_\xi| = \int_{\Om_\xi} \sg_3 \De v_0 = \int_{\pa\Om_\xi} \sg_3 \pa_{n_\xi} v_0 = \int_{\pa\Om_2} \sg_3 \pa_{n_\xi} v_0\circ (\id+\xi n) J_\tau(\xi).
     \end{aligned}
\end{equation*}
Moreover, since the tangential Jacobian $J_\tau$ never vanishes, it is clear that, by definition, $\Psi(\xi,\eta)=0$ if and only if $v_{\xi,\eta}$ satisfies \eqref{what we want}.

In what follows, we will apply the implicit function theorem (Theorem \ref{ift}) to the mapping $\Psi: C_*^{2,\al}(\pa\Om_2)\times C^{2,\al}(\pa\Om_1)\to C_*^{1,\al}(\pa\Om_2)$. To this end, it will be sufficient to show that the partial Fr\'echet derivative of $\Psi$ with respect to $\xi$ at $\xi=0$ is a bijection between $C_*^{2,\al}(\pa\Om_2)$ and $C_*^{1,\al}(\pa\Om_2)$. A direct computation with \eqref{u_0}, \eqref{v_0}, \eqref{v' eq} and \eqref{w xi} at hand yields:
\begin{equation*}
\begin{aligned}
       \pa_\xi \Psi(0,0)[\xi] &= \Big(\underbrace{\pa_n u_0-\sg_3 \pa_n v_0}_{=0}  \Big) J_\tau'(0)[\xi] 
+ \bigg(\pa_n v'[\xi] + \Big( \underbrace{(\pa_n)^2 u_0 - \sg_3 (\pa_n)^2 v_0}_{=0} \Big)\xi       \bigg) \underbrace{J_\tau(0)}_{=1}\\
&=\pa_n v'[\xi]= \frac{1}{N} \left(1-\frac{1}{\sg_3}\right)\cN(\xi).
\end{aligned}
\end{equation*}
We remark that the computation above is also dramatically simplified  because the Fr\'echet derivative of $n_\xi\circ (\id+\xi n)$ at $\xi=0$ is tangent to $\pa\Om_2$, that is, orthogonal to $n$ (as a matter of fact, it is equal to $-\gr_\tau \xi$ by \cite[Proposition 5.4.14]{HP2018}). 
Finally, since $\sg_3\ne \sg_2=1$, the conclusion follows from Lemma \ref{N is a bijection}.
\end{proof}

\section{Some final comments}\label{s6}
In this section, we give some comments on the various topological and regularity assumptions used in this paper. 
\subsection*{On the topological assumptions in Theorem \ref{thm bessatsu}}
Theorem \ref{thm bessatsu} of page \pageref{thm bessatsu} ensures spherical symmetry in a two-phase setting under the topological assumption that $\Om_1$ has finitely many connected components and $D_1:= \Om_2\setminus\ol\Om_1$ is connected.  
Notice that, by Theorem \ref{mainthm II}, we know that the connectedness of $D_1$ is necessary. 
Indeed, since Theorem \ref{mainthm II} holds in the ``two-phase-three-layer" case (that is, $m=3$ and $\sg_1=\sg_3$), there exists a non-trivial triplet of domains $D_\eta\subset\subset \Om_\xi \subset\subset \Om_3$ such that \eqref{infinite overdetermination} holds. Renaming the domains as 
\begin{equation*}
    \Om_2:=\Om_3, \quad \Om_1:= \Om_\xi\setminus\ol D_\eta,\quad D_2:= \Om_2\setminus \ol \Om_1 = D_\eta \cupdot (\Om_3\setminus \ol \Om_\xi)
\end{equation*}
gives the desired counterexample to Theorem \ref{thm bessatsu} for disconnected $D_2=\Om_2\setminus\ol\Om_1$. 
To the best of my knowledge, it is still an open question whether there exists a counterexample to Theorem where $\Om_1$ has infinitely many connected components without developing a microstructure. 

\subsection*{On the regularity}
In Theorem \ref{mainthm II} we constructed a pair of nontrivial domains $D_\eta\subset\subset \Om_\xi$ of class $C^{2,\al}$ such that \eqref{infinite overdetermination} is satisfied. We remark that this particular choice of regularity has been made only to simplify the exposition. Indeed, one could have chosen higher regularity spaces as $C^{k,\al}$ ($k\ge 3$) in \eqref{Psi} or even different regularities altogether for $\eta$ and $\xi$. The latter can be done by the ``simultaneous asymmetric perturbation method" introduced in \cite{cava sap}

\subsection*{On the choice of the mapping $\Psi$}
We remark that the choice of the mapping $\Psi$ used in the proof of Theorem \ref{mainthm II} is not by chance. The reader might wonder why we opted for such a convoluted approach (cutting the solution at the second phase, deforming it, and then gluing it back together with the radial unperturbed solution) instead of the more direct approach given by a simple Neumann-tracking on $\pa\Om_m$. In what follows, we aim to give an intuitive explanation of why such a naive method fails. 
Instead of the one defined in \eqref{Psi}, let $\Psi$ be the following Neumann-tracking type operator: 
\begin{equation*}
\begin{aligned}
 \Psi:\; & C_*^{2,\al}(\pa\Om_2)\times C^{2,\al}(\pa\Om_1) \longrightarrow C_*^{k,\al}(\pa\Om_m),\\
&(\xi, \eta) \mapsto \pa_n u_{\xi,\eta}-c_1,
\end{aligned}
\end{equation*}
where $u_{\xi,\eta}$ is the solution to \eqref{bvp} with respect to \eqref{nontrivial domains}.
First of all, notice that, by the Schauder regularity theory, $u_{\xi,\eta}$ is of class $C^{k,\al}$ (for all $k\ge 2$) in a neighborhood of $\pa\Om_m$ and so the map $\Psi$ is well defined.
Accordingly, one has to replace the Dirichlet-to-Neuman map $\cN$ with the following ``jump-to-Neumann" map $\xi\mapsto \cJ(\xi):=\pa_n w[\xi]$, where $w[\xi]$ is the solution to the following transmission problem:  
\begin{equation*}
    \begin{cases}
        -\De w = 0 \quad \tin D_1 \cup \dots \cup D_m,\\
        \jump{w}=\jump{\sg \pa_n w}=0 \quad \ton \pa \Om_i\; (i=1,3,4,\dots, m-1),\\
        \jump{w}=\xi \quad \ton \pa\Om_2,\\
        \jump{\sg \pa_n w}=0 \quad \ton \pa\Om_2,\\
        w=0 \quad \ton \pa\Om_m.
    \end{cases}
\end{equation*}
As briefly mentioned before, notice that the function $\pa_n w[\xi]$ is arbitrarily smooth, irrespective of the regularity of $\xi$. In other words, in passing from $\xi$ to $\cJ(\xi)$, all information about the regularity of $\xi$ gets lost and thus solving the equation 
\begin{equation*}
    \cJ(\xi):=\pa_n w[\xi] =\eta
\end{equation*}
in the appropriate Banach spaces becomes an ill-conditioned problem. As a result, Fredholmness is lost and the proofs of the analogs of Lemmas \ref{id+N is a bijection}--\ref{N is a bijection} fail. 
As a rule of thumb, we can say that this sort of ill-conditioning usually happens when the ``free boundary" (in this case $\pa\Om_\xi$) and the ``overdetermined boundary" (that is, the boundary where the tracking takes place, in this case,  $\pa\Om_m$) do not coincide. 

\begin{small}

\end{small}

\bigskip
\bigskip

\noindent
\textsc{
Mathematical Institute, Graduate School of Science, Tohoku University, Aoba-ku, 
Sendai 980-8578, Japan}\\
\noindent
{\em Electronic mail address:}
cavallina.lorenzo.e6@tohoku.ac.jp

\end{document}